\documentclass[amsfonts,12pt]{amsart}

\usepackage{epsfig}
\newtheorem{thm}{Theorem}[section]

\newtheorem{lem}[thm]{Lemma}

\newcommand{\inte}{{\mathrm{int}}\,}

\newcommand{\sgn}{{\mathrm{sgn}}\,}

\newcommand{\ee}{\varepsilon}

\newcommand{\R}{{\Bbb R}}

\begin{document}
\hfill\today
\bigskip

\title{A problem of Klee on inner section functions of convex bodies}
\author[R.~J.~Gardner, D.~Ryabogin, V.~Yaskin, and A.~Zvavitch]
{R.~J.~Gardner, D.~Ryabogin, V.~Yaskin, and A.~Zvavitch}
\address{Department of Mathematics, Western Washington University,
Bellingham, WA 98225-9063} \email{Richard.Gardner@wwu.edu}
\address{Department of Mathematics, Kent State University,
Kent, OH 44242, USA} \email{ryabogin@math.kent.edu}
\address{Department of Mathematical and
Statistical Sciences, University of Alberta, Edmonton, Alberta T6G 2G1, Canada} \email{vladyaskin@math.ualberta.ca}
\address{Department of Mathematics, Kent State University,
Kent, OH 44242, USA} \email{zvavitch@math.kent.edu}
\thanks{First author supported in
part by U.S.~National Science Foundation Grant DMS-0603307. Second and fourth authors supported in part by U.S.~National Science Foundation Grant DMS-0652684. Third author supported in part by NSERC. This research arose from discussions at the workshop on Mahler's Conjecture and Duality in Convex Geometry at the American Institute of Mathematics, August 9--13, 2010.}
 \maketitle
\begin{abstract}
In 1969, Vic Klee asked whether a convex body is uniquely determined (up to translation and reflection in the origin) by its inner section function, the function giving for each direction the maximal area of sections of the body by hyperplanes orthogonal to that direction. We answer this question in the negative by constructing two infinitely smooth convex bodies of revolution about the $x_n$-axis in $\R^n$, $n\ge 3$, one origin symmetric and the other not centrally symmetric, with the same inner section function.  Moreover, the pair of bodies can be arbitrarily close to the unit ball.
\end{abstract}

\section{Introduction}
Let $K$ be a convex body in $\mathbb R^n$.  The {\em inner section function} $m_K$ is defined by
$$m_K(u)=\max_{t\in \R} V(K\cap (u^{\perp}+tu)),$$
for $u\in S^{n-1}$.  Here $u^{\perp}$ denotes the hyperplane through the origin orthogonal to $u$ and $V$ denotes volume, that is, the $k$-dimensional Hausdorff measure of a $k$-dimensional body.  (See Section~\ref{prelim} for other notation and definitions.)  Thus the inner section function simply gives, for each direction, the maximal area of cross-sections orthogonal to that direction.

Interest in the inner section function goes back at least to the 1926 paper \cite{B} of Bonnesen.  It has been called the {\em inner $(n-1)$-quermass}, particularly by authors who prefer to use the term ``outer $(n-1)$-quermass" for the brightness function, the function giving the areas of the orthogonal projections of a body onto hyperplanes.  Values of the inner section function have also been called {\em HA-measurements}, a term devolving from the de Haas-van Alphen effect in the study of Fermi surfaces of metals.  The following explanation is taken from \cite[Note~8.12]{G}. The Fermi surface of a metal bounds a body formed, in velocity space, by velocity states occupied at absolute zero by valence electrons of the metal.  The Pauli exclusion principle allows no more than two electrons (with opposite spins) to possess the same velocity (i.e., speed and direction), so electrons can only move into unoccupied states lying outside the Fermi surface.  The more
electrons there are near the Fermi surface, the larger the number that can increase their energy when the metal is heated and the larger the number whose spins can be aligned with a magnetic field. In this way the Fermi surface relates to the specific heat and magnetic properties of the metal, and the concept also provides an explanation of conductivity and ductility, for example.  The body bounded by a Fermi surface is not generally convex.  However, the definition of the inner section function above extends naturally to any bounded Borel set and for the body bounded by a Fermi surface it may actually be measured by means of the de Haas--van Alphen effect, i.e., magnetism induced in the metal by a strong magnetic field at a low temperature.

In two influential articles published about 40 years ago, Klee \cite{Kl}, \cite{Kl2} asked whether a convex body in $\R^n$, $n\ge 3$, is uniquely determined, up to translation and reflection in the origin, by its inner section function.  The problem also appears in the books \cite[p.~24]{CFG} and \cite[Problem 8.8(i)]{G}. It has long been known that planar convex bodies whose inner section functions are constant are precisely the planar convex bodies of constant width.  This explains the restriction $n\ge 3$ in Klee's problem, since when $n=2$, any convex body of constant width that is not a disk provides a counterexample.

The main purpose of the present paper is to answer Klee's question negatively.  We do this by constructing two convex bodies in $\R^n$, $n\ge 3$, each a body of revolution about the $x_n$-axis, such that one body is origin symmetric and the other is not centrally symmetric, while both have the same inner section function.

Klee's problem is now seen as belonging to a fairly extensive literature around the important concept of an intersection body.  Given a star body $L$ in $\R^n$, i.e., a set star-shaped with respect to the origin whose radial function $\rho_L$ is continuous, the {\em intersection body} $IL$ of $L$ is defined by
$$\rho_{IL}(u)=V(L\cap u^{\perp}),$$
for $u\in S^{n-1}$.  Clearly, $IL$ is an origin-symmetric star body.  Moreover, it is known that an origin-symmetric star body $L$ is uniquely determined by the values $V(L\cap u^{\perp})$, $u\in S^{n-1}$, and hence by its intersection body $IL$.  This was proved by Funk \cite{F} for convex bodies in $\R^3$; the general result is called Funk's section theorem in \cite[Corollary~7.2.7]{G}, though it was not stated explicitly in this generality until the work of Lifshitz and Pogorelov \cite{LP} on Fermi surfaces.  Intersection bodies were introduced by Lutwak \cite{L5} and turned out to be the key to solving the Busemann-Petty problem (see \cite{G3}, \cite{GKS}, and \cite{Z1}, as well as \cite[Section~8.2]{G} and \cite[Chapter 5]{Ko}).

For an origin-symmetric convex body $K$ in $\R^n$, the Brunn-Minkowski inequality (see \cite{Gar02}) implies that $m_K(u)=V(K\cap u^{\perp})$ for each $u\in S^{n-1}$.   It follows immediately that Klee's problem has an affirmative answer if set entirely within the class of origin-symmetric convex bodies.  It is also known that in this case $IK$ is itself an origin-symmetric convex body, by Busemann's theorem; see, for example, \cite[Theorem~8.1.10]{G}.

Now suppose that $K$ is an arbitrary convex body in $\R^n$.  The {\em cross-section body} $CK$ of $K$ is defined by
$$\rho_{CK}(u)=m_K(u),$$
for $u\in S^{n-1}$.  Cross-section bodies, introduced by H.~Martini, are the subject of \cite[Section~8.3 and Note~8.12]{G}.  As we have seen, when $K$ is origin symmetric, $m_K(u)=V(K\cap u^{\perp})$ for each $u\in S^{n-1}$, and hence $CK=IK$ is also an origin-symmetric convex body.

With this background, Klee's question can be rephrased as asking whether a convex body in $\R^n$, $n\ge 3$, is uniquely determined, up to translation and reflection in the origin, by its cross-section body.  In these terms, our solution involves constructing a non-centrally-symmetric convex body $K$ whose cross-section body $CK$ is the (necessarily origin-symmetric and convex) intersection body $IL$ of some origin-symmetric convex body $L$.  Thus $K$ and $L$ are not equal, up to translation and reflection in the origin, while $CK=IL=CL$.

The paper is organized as follows.  After the preliminary Section~\ref{prelim}, the main Section~\ref{main} states our result and proves it via a succession of lemmas.  The paper ends with the short Section~\ref{conclude} of concluding remarks.

\section{Definitions, notation, and preliminaries}\label{prelim}

As usual, $S^{n-1}$ denotes the unit sphere and $o$ the origin in Euclidean
$n$-space ${\Bbb R}^n$.  The unit ball in $\R^n$ will be denoted by $B^n$. The
standard orthonormal basis for $\R^n$ will be $\{e_1,\dots,e_n\}$.  The inner product of vectors $x$ and $y$ is denoted by $\langle x,y\rangle$ and the Euclidean norm of $x$ by $|x|$.

If $X$ is a set,  we denote by $\partial X$ and $\inte X$ the {\it
boundary} and {\it interior} of $X$, respectively.

The set $-X$ is the {\em reflection} of $X$ in the origin.  A set $X$ is {\it origin symmetric} if $X=-X$ and {\em centrally symmetric} if it is a translate of an origin-symmetric set.

A {\it body} is a compact set equal to the closure of its interior.  If $K$ is a $k$-dimensional body in $\R^n$, then $V(K)$ is its {\em volume} ${\mathcal{H}}^k(K)$, where ${\mathcal{H}}^k$ is $k$-dimensional Hausdorff measure in $\R^n$, $k=1,\dots,n$. The notation $dz$ will always
mean $d{\mathcal{H}}^k(z)$ for the appropriate $k=1,\dots, n$.
We follow Schneider \cite{Sch93} by writing $\kappa_n$ and $\omega_n=n\kappa_n$ for the volume and surface area of the unit ball in $\R^n$, respectively.

A {\it convex body} is a compact convex set with nonempty interior.

If $L$ is a body containing the origin in its interior and star-shaped with respect to the origin, its {\em radial function} $\rho_L$ is defined by
$$\rho_L(x)=\max\{c\in \R: cx\in L\},$$
for $x\in\R^n\setminus\{o\}$.  Note that $\rho_L(x)=1$ if and only if $x\in \partial L$.  For $u\in S^{n-1}$, $\rho_L(u)$ gives the distance from the origin to $\partial L$ in the direction $u$.  Moreover, $\rho_L$ is homogeneous of degree $-1$, i.e., $\rho(tx)=\rho(x)/t$ for $t>0$.  The homogeneity means that we can often regard $\rho_L$ as a function on $S^{n-1}$.

For the purposes of this paper, a {\em star body} is a body whose radial function is continuous on $S^{n-1}$.  (The reader should be aware that other definitions are often used.)  When considering star bodies of revolution about the $x_n$-axis in $\R^n$, we can regard the radial function as a function of the spherical polar coordinate angle $\phi$ with the positive $x_n$-axis.

Let $K$ be a convex body and let $u\in S^{n-1}$.  The {\em parallel section function} $A_{K,u}(t)$ is defined by
\begin{equation}\label{Adef}
A_{K,u}(t)=V(K\cap (u^{\perp}+tu)),
\end{equation}
for $t\in \R$.  See \cite[Lemma~2.4 and Section~3.3]{Ko} for recent results concerning this function.  With this notation, the inner section function of $K$ is
$$m_K(u) = \max_{t\in \mathbb R} A_{K,u}(t),$$
for $u\in S^{n-1}$.

In the literature, it is often ignored that the parallel section function $A_{K,u}$ is just the $(n-1)$-dimensional X-ray of $K$ in the direction $u$.  See \cite[Chapter~2]{G}.  Thus this function has a prominent role in Geometric Tomography, and indeed Klee's problem was a significant open question in this subject.

As usual, $C(X)$ and $C^{k}(X)$, $1\le k\le\infty$, denote the classes of continuous and $k$ times continuously differentiable functions on a subset $X$ of $\R^n$, and $C_e(X)$ and  $C^k_e(X)$ will signify the even functions in these classes.

We denote by $R$ the {\it spherical Radon transform}, defined by
$$(Rf)(u)=\int_{S^{n-1}\cap u^{\perp }}f(v)\,dv,$$ for bounded Borel
functions $f$ on $S^{n-1}$.  It is known (see \cite[Theorem~C.2.4]{G}) that $R$ is injective on the class of even functions on $S^{n-1}$, and this fact, applied to $(n-1)$th powers of radial functions of star bodies, immediately yields Funk's section theorem mentioned in the Introduction.  As a historical aside, we remark that this injectivity property of $R$ was proved in 1904 by Minkowski (see \cite[p.~430]{G}), who used it to show that a convex body in $\R^3$ of constant girth also has constant width. However, Minkowski apparently did not notice the result of Funk (who proved it differently).

We adopt a standard definition of the Fourier transform $\widehat{f}$ of a
function $f\in L_1(\R^n)$, namely
$$\widehat{f}(x)=\int_{\R^n}f(y)e^{-i\langle x,y\rangle}\,dy.$$

If $f\in C(S^{n-1})$ (or $f\in C^k(S^{n-1})$) and $p\in \R$, then $f$ can be extended by setting
\begin{equation}\label{homog}
f(x)=|x|^{-n+p}f\left(\frac{x}{|x|}\right),
\end{equation}
for $x\in \R^n\setminus\{o\}$, and this extension is a homogeneous of degree $-n+p$ function in $C(\R^n\setminus \{o\})$ (or $C^k(\R^n\setminus \{o\})$, respectively). Note that throughout the paper, we shall use the same notation for a function on $S^{n-1}$ and its homogeneous extension to $\R^n\setminus\{o\}$, since the distinction will be clear from the context. The Fourier transform of this extension of $f$, which we also denote by $\widehat{f}$, exists in the sense of distributions, as explained in \cite[p.~389]{GYY} or \cite[Sections~2.5 and~3.3]{Ko}, for example.  Where appropriate, statements in the sequel involving Fourier transforms of homogeneous extensions of functions on $S^{n-1}$ are also to be considered in the sense of distributions.

Let $D^{\alpha}$, where $\alpha=(\alpha_1,\dots,\alpha_n)$ is a multiindex of nonnegative integers, be the differential operator defined by
$$D^{\alpha}f=\frac{\partial^{|\alpha|_1}f}{\partial x_1^{\alpha_1}\cdots\partial x_n^{\alpha_n}},$$
for $f\in C^{|\alpha|_1}(\R^n)$, where $|\alpha|_1=\sum_{j=1}^n\alpha_j$.  Then
\begin{equation}\label{FD1}
\left(D^{\alpha}f\right)^{\wedge}=i^{|\alpha|_1}x_1^{\alpha_1}\cdots x_n^{\alpha_n}\widehat{f}
\end{equation}
and
\begin{equation}\label{FD2}
D^{\alpha}\widehat{f}=(-i)^{|\alpha|_1}\left(x_1^{\alpha_1}\cdots x_n^{\alpha_n}f\right)^{\wedge}.
\end{equation}
These formulas also hold in the distributional sense; see, for example, \cite[p.~35]{Ko}.
The relation
\begin{equation}\label{Lap}
\left(\triangle f\right)^{\wedge}=-|x|^2\widehat{f},
\end{equation}
where $\triangle$ is the Laplacian operator in $\R^n$, follows directly from (\ref{FD1}).

If $f$ is even, then $\widehat{f}$ is also even.  If $f\in C_e(\R^n\setminus \{o\})$ is homogeneous of degree $-n+1$, then by \cite[Lemma~3.7]{Ko}, $\widehat{f}$ is a homogeneous of degree $-1$ function in $C_e(\R^n\setminus \{o\})$ and
\begin{equation}\label{RF}
(Rf)(u)=\frac1\pi\widehat{f}(u),
\end{equation}
for $u\in S^{n-1}$.

The spherical Radon transform is a continuous bijection from $C_e^{\infty}(S^{n-1})$ to itself; see, for example, \cite[Theorem~C.2.5]{G}.  Let $g\in C_e^{\infty}(S^{n-1})$, and extend $g$ as in (\ref{homog}) (with $f$ there replaced by $g$ and $p=n-1$) to a homogeneous of degree $-1$ function in $C_e^{\infty}(\R^n\setminus \{o\})$.  If $f=\widehat{g}$, then, by \cite[Lemma~3.16]{Ko},  $f\in C_e^{\infty}(\R^n\setminus \{o\})$ is homogeneous of degree $-n+1$.  Since $g$ is even, we have $\widehat{\widehat g}=(2\pi)^n g$ and from (\ref{RF}) it follows that
\begin{equation}\label{RinvF}
(R^{-1}g)(u)=\frac{1}{(2\pi)^n}(R^{-1}\widehat{\widehat{g}}\,)(u)=
\frac{1}{(2\pi)^n}(R^{-1}\widehat{f}\,)(u)=\frac{\pi}{(2\pi)^n}f(u)
=\frac{\pi}{(2\pi)^n}\widehat{g}(u),
\end{equation}
for all $u\in S^{n-1}$.

A {\em rotationally symmetric} function $f$ on $S^{n-1}$ is one that can be defined via a function $f(\phi)$ of the vertical angle $\phi$ in spherical polar coordinates by setting $f(u)=f(\arccos u_n)=f(\phi)$ for $u=(u_1,\dots,u_n)\in S^{n-1}$. Moreover, we will consider $f$ as an even function on $S^1$, by first extending $f$ to $[-\pi,\pi]$ by letting $f(-\phi)=f(\phi)$ and then regarding $S^1$ as $[-\pi,\pi]$ with its endpoints identified.
Of course we are abusing notation by using the same letter for the function of $u\in S^{n-1}$, the function of $\phi\in [0,\pi]$, and the function of $\phi\in S^1$, but will adopt this convenient practise throughout the paper.

\section{Main result}\label{main}

\begin{thm}\label{mainthm}
There exist convex bodies $K$ and $L$ in $R^n$, $n\ge 3$, with $\rho_K,\rho_L\in C^{\infty}(S^{n-1})$, each a body of revolution about the $x_n$-axis, such that $K$ is not centrally symmetric, $L$ is origin symmetric, and $m_K=m_L$.
\end{thm}

The theorem will be proved in a series of lemmas.
\begin{lem}\label{lem1}
Let $0<\ee<1$ and let $K_{\ee}$ be the body of revolution about the $x_n$-axis in $\R^n$ with radial function given in spherical polar coordinates by
\begin{equation}\label{rho}
\rho_{K_{\ee}}(\phi)=\left(1+\ee \cos^3 \phi\right)^{-1/3},
\end{equation}
where $0\le \phi\le \pi$ is the angle with the positive $x_n$-axis, or equivalently by
\begin{equation}\label{rho2}
\rho_{K_\ee}(x)=(|x|^3+\ee x_n^3)^{-1/3},
\end{equation}
where $x=(x_1,\dots,x_n)\in \R^n\setminus\{o\}$.  Then $\rho_{K_{\ee}}\in C^{\infty}(S^{n-1})$, the body $K_{\ee}$ is not centrally symmetric, and there is a positive $\ee_0\le 1$ such that $K_{\ee}$ is convex for all $\ee<\ee_0$.
\end{lem}

\begin{proof}
It is clear from its definition via (\ref{rho2}) that $\rho_{K_{\ee}}\in C^{\infty}(S^{n-1})$ provided $\ee<1$.  To prove that $K_{\ee}$ is not centrally symmetric and is convex for sufficiently small $\ee$, it suffices to show that this is true for the planar convex body $J_{\ee}$ obtained by intersecting $K_{\ee}$ with the $x_1,x_n$-plane.  The radial function of $J_{\ee}$ is also given by (\ref{rho}), where we can regard $\phi$ as the polar coordinate angle in the $x_n,x_1$-plane (note that $\rho_{J_{\ee}}(-\phi)=\rho_{J_{\ee}}(\phi)$).  It is well known that the curvature of a planar $C^2$ curve given in polar coordinates by $r=r(\theta)$ is
\begin{equation}\label{cpolar}
\frac{2(r')^2-rr''+r^2}{\left((r')^2+r^2\right)^{3/2}}.
\end{equation}
Using this, we find that the curvature of $J_{\ee}$ is
$$\frac{(1+\ee\cos^3\phi+2\ee\sin^2\phi\cos\phi)(1+\ee\cos^3\phi)^{4/3}}{
((1+\ee\cos^3\phi)^2+\ee^2\cos^4\phi\sin^2\phi)^{3/2}}.$$
When $\ee$ is sufficiently small, the curvature is positive and hence $J_{\ee}$ is convex.  Also, the curvatures when $\phi=0$ and $\phi=\pi$ are $(1+\ee)^{-2/3}$ and $(1-\ee)^{-2/3}$, respectively.  It follows that $J_{\ee}$ is not centrally symmetric, since a center of symmetry would have to lie on the $x_n$-axis, and $\phi=0$ and $\phi=\pi$ correspond to antipodal points in $\partial J_{\ee}$.
\end{proof}

The body $K$ in Theorem~\ref{mainthm} will be $K=K_{\ee}$ for a suitable $\ee$.  From now on we assume that $\ee<\ee_0$.  It can be seen numerically that one can take $\ee_0=0.91$, but note that in any case $\ee_0< 1$ by the definition of $K_{\ee}$.  Since $K_\ee$ is a body of revolution about the $x_n$-axis, for any fixed $t\in\R$, the function $A_{K_{\ee},u}(t)$, $u\in S^{n-1}$ defined by (\ref{Adef}) is rotationally symmetric.  Therefore we can write $A_{K_{\ee},u}(t)=A_{K_{\ee},\,\phi}(t)$,
where we can either consider $\phi\in [0,\pi]$ as the angle between $u$ and $e_n$ or as an element $\phi\in S^1$, as explained at the end of Section~\ref{prelim}.

\begin{lem}\label{lem2}
Let $n\ge 3$ and let $K_{\ee}$ be as in Lemma~\ref{lem1}.  Then

$\mathrm{(i)}$  For each $\phi\in S^1$, the parallel section function $A_{K_{\ee},\,\phi}(t)$ has a maximum at a unique point $t=t_\ee (\phi)$.

$\mathrm{(ii)}$ There is a positive $\ee_1(n)\le \ee_0$ such that
\begin{equation}\label{t}
t_\ee(\phi)=T_{\ee}(\phi)\ee,
\end{equation}
where $T_{\ee}(\phi)\in C^{\infty}(S^1)$ for all $0<\ee<\ee_1(n)$.

$\mathrm{(iii)}$ For $k=0,1,\dots$, there is a constant $c_1(k,n)$ such that
\begin{equation}\label{ds}
\left|\frac{d^k T_{\ee}(\phi)}{d\phi^k}\right|\le c_1(k,n),
\end{equation}
for all $0<\ee<\ee_1(n)$ and $\phi \in S^1$.
\end{lem}

\begin{proof}
$\mathrm{(i)}$  The Brunn-Minkowski inequality (see \cite{Gar02} or \cite[Section~B.2]{G}) implies that for each $\phi$ the function $A_{K_{\ee},\,\phi}(t)^{1/(n-1)}$ is concave and hence $A_{K_{\ee},\,\phi}(t)$ is unimodal, i.e., increases and then decreases with $t$.  Therefore, if its maximum is not attained at a unique point, it is attained at all points in a closed interval $[t_1,t_2]$, say, $t_1\neq t_2$.  Now the equality condition of the Brunn-Minkowski inequality implies that the sections $K_{\ee}\cap (u^{\perp}+tu)$, $t\in [t_1,t_2]$ are all homothetic and hence, by convexity, the union of these sections is a (possibly slanted) cylinder.  Since $K_{\ee}$ is a body of revolution, it is easy to see that this is only possible when the cylinder is a right spherical cylinder with its axis along the $x_n$-axis.  However, this implies that $K_{\ee}$ has vertical line segments in its boundary, which clearly contradicts its definition via (\ref{rho}).

$\mathrm{(ii)}$  Let $u\in S^{n-1}$ and $t\in \mathbb R$ be such that $tu\in\inte K_{\ee}$. Note that for any $v\in  S^{n-1}\cap u^\perp$, we have $\rho_{K_\ee - tu}(v)v\in \partial(K_\ee - tu)$ and hence $x=tu +\rho_{K_\ee - tu}(v)v\in \partial K_\ee$.
Then $\rho_{K_{\ee}}(x)=1$ and hence (\ref{rho2}) yields
\begin{equation}\label{eqn_for_rho}
1= ({t^2+\rho_{\ee}^2})^{3/2} + \ee(t\cos\phi +\rho_{\ee} v_n)^3,
\end{equation}
where $\phi\in [0,\pi]$, $v=(v_1,\dots,v_n)$, and where for brevity we write $\rho_{\ee}=\rho_{K_\ee - tu}(v)$.  If we set
\begin{equation}\label{en}
e_n=u\cos\phi+w\sin\phi,
\end{equation}
where $w\in S^{n-1}\cap u^\perp$, and $s=\langle v,w\rangle$, then
$$
v_n=\langle v, e_n\rangle=\langle v, w\rangle\sin\phi=s\sin\phi.
$$
Then we can rewrite (\ref{eqn_for_rho}) as
\begin{equation}\label{rhonew}
1= ({t^2+\rho_{\ee}^2})^{3/2} + \ee(t\cos\phi +\rho_{\ee} s\sin\phi)^3,
\end{equation}
where $\phi\in [0,\pi]$ and $-1\le s\le 1$.  Thus the radial function of $(K_{\ee}-tu)\cap u^{\perp}=K_{\ee}\cap (u^{\perp}+tu)$, as a function $\rho_{\ee}=\rho_{\ee}(t,\phi,s)$, is implicitly defined by (\ref{rhonew}).  We can also regard $\rho_{\ee}$ as defined for $\phi\in S^1$ by setting
\begin{equation}\label{rhoeven}
\rho_{\ee}(t,\phi,s)=\rho_{\ee}(t,-\phi,-s),
\end{equation}
for $\phi\in [-\pi,0]$.

We claim that there is a positive $\delta_1\le \ee_0$ such that for $0<\ee<\delta_1$, we have $|t_{\ee}(\phi)|<1/4$ and the function $\rho_{\ee}=\rho_{\ee}(t,\phi,s)$ implicitly defined by (\ref{rhonew}) is infinitely differentiable with respect to each variable on $(-1/4,1/4)\times S^1\times [-1,1]$. (At the endpoints $s=-1$ and $s=1$ of $[-1,1]$ we take one-sided derivatives.)  To see this, note first that by (\ref{rho}), $K_{\ee}\rightarrow B^n$ in the Hausdorff metric as $\ee\rightarrow 0+$.  Consequently, there is a positive $\delta_0\le \ee_0$ such that for $0<\ee<\delta_0$, we have $|t_{\ee}(\phi)|<1/4$ and $3/4<\rho_{K_{\ee}}(\phi)<5/4$ for all $\phi\in S^1$.  Then $1/2<\rho_{\ee}=\rho_{K_\ee - tu}(v)<3/2$ when $|t|<1/4$ and $0<\ee<\delta_0$. Let $F_{\ee}:\R^4\rightarrow \R$ be defined by
\begin{equation}\label{FF}
F_{\ee}(x_1,x_2,x_3,y)=(x_1^2+y^2)^{3/2}+\ee(x_1\cos x_2+yx_3\sin x_2)^3-1.
\end{equation}
Then we can choose $0<\delta_1\le \delta_0$ small enough to ensure that if $0<\ee<\delta_1$, then
\begin{equation}\label{FFd}
\frac{\partial F_{\ee}}{\partial y}=3(x_1^2+y^2)^{1/2}y+3\ee(x_1\cos x_2+yx_3\sin x_2)^2x_3\sin x_2> 3/4-3(7/4)^2\ee>0,
\end{equation}
for all $(x_1,x_2,x_3,y)\in (-1/4,1/4)\times[-\pi, \pi]\times[-1,1]\times (1/2,3/2)$.  Comparing (\ref{rhonew}) and (\ref{FF}) and noting that $F_{\ee}$ is periodic with period $2\pi$ in the variable $x_2$, we see that the implicit function theorem implies that for all $0<\ee<\delta_1$, there is a unique function $\rho_{\ee}=\rho_{\ee}(t,\phi,s)$ that satisfies (\ref{rhonew}) and is infinitely differentiable with respect to each variable on $(-1/4,1/4)\times S^1\times [-1,1]$. This proves the claim.

Differentiation of (\ref{rhonew}) with respect to $t$ yields
\begin{equation}\label{partrho}
\frac{\partial \rho_{\ee}}{\partial t}=-\frac{(t^2+\rho_{\ee}^2)^{1/2} t + \ee(t \cos\phi +\rho_{\ee}  s \sin\phi)^2\cos\phi}{(t^2+\rho_{\ee}^2)^{1/2} \rho_{\ee} + \ee(t \cos\phi +\rho_{\ee}   s \sin\phi)^2  s\sin\phi}.
\end{equation}
Now $\rho_{\ee}$, $s$, and $t$ are bounded above uniformly in $\ee$ for $0<\ee<\delta_0$.  Also, since $\rho_{\ee}>1/2$ when $|t|<1/4$, the presence of the term $(t^2+\rho_{\ee}^2)^{1/2} \rho_{\ee}$ in the previous denominator guarantees that this denominator is bounded below by a positive constant, uniformly in $\ee$ for $0<\ee<\delta_1$. (We can use $\delta_1$ here in view of (\ref{FFd}).) Moreover, the same denominator appears in $\partial \rho_{\ee}/\partial \phi$ and $\partial \rho_{\ee}/\partial s$, and powers of it appear in all higher derivatives with respect to $t$, $\phi$, and $s$.  Therefore the partial derivatives of $\rho_{\ee}$ of any order are each bounded uniformly in $\ee$ for $0<\ee<\delta_1$.

Next, we observe that if $\alpha$ is the angle between $v$ and the $w$ defined by (\ref{en}), then the set of points in $S^{n-1}\cap u^{\perp}$ having the same fixed angle $\alpha$ at $o$ is an $(n-3)$-sphere of radius $\sin\alpha$.  Then $s=\sin\alpha$ and since $A_{K_{\ee},\phi}(t) = A_{K_{\ee}-tu,\phi}(0)$, we have
\begin{eqnarray}\label{A}
A_{K_{\ee},\phi}(t) =V((K_{\ee}-tu)\cap u^{\perp})&=&\frac{1}{n-1} \int_{S^{n-1}\cap u^\perp} \rho_{\ee}(t,\phi,\cos\alpha)^{n-1}\,dv\nonumber\\
&=&\frac{\omega_{n-2}}{n-1} \int_{0}^{\pi} \rho_{\ee}(t,\phi, \cos\alpha)^{n-1}\sin^{n-3}\alpha\,d\alpha\nonumber\\
&=&\frac{\omega_{n-2}}{n-1}\int_{-1}^1\rho_{\ee}(t,\phi, s)^{n-1}\,d\mu(s),
\end{eqnarray}
where $d\mu(s) = (1-s^2)^{(n-4)/2} ds$. (The geometry behind the substitution in the integral is depicted in \cite[Figure~8.4, p.~314]{G}; compare the proof of \cite[Theorem~C.2.9]{G}.  See also \cite[Lemma~1.3.1(ii)]{Gr}.)   Recalling that
the integrand in (\ref{A}) can be regarded as defined for $\phi\in S^1$ via (\ref{rhoeven}), we see that we can regard $A_{K_{\ee},\phi}(t)$ as an even function of $\phi\in S^1$ in the sense that $A_{K_{\ee},-\phi}(t) = A_{K_{\ee},\phi}(t)$ for all $\phi\in S^1$.

The restriction $n\ge 3$ allows differentiation under the integral sign in integrals with respect to $d\mu(s)$ such as (\ref{A}) for which the integrand is continuously differentiable with respect to its variables.  We shall use this several times in the sequel without further comment.

At the unique point $t=t_\ee(\phi)$ where $A_{K_{\ee},\phi}(t)$ attains its maximum, we have, by differentiating (\ref{A}) with respect to $t$ and using (\ref{partrho}), that
\begin{eqnarray}\label{first_deriv}
A_{K_{\ee},\phi}'(t)&=&\omega_{n-2}\int_{-1}^1 \rho_{\ee}(t,\phi,s)^{n-2} \frac{\partial \rho_{\ee}}{\partial t}(t,\phi,s) \,d\mu(s)\nonumber\\
&=&-\omega_{n-2}\int_{-1}^1 \rho_{\ee}^{n-2}\frac{(t^2+\rho_{\ee}^2)^{1/2} t + \ee(t \cos\phi +\rho_{\ee}  s \sin\phi)^2\cos\phi}{(t^2+\rho_{\ee}^2)^{1/2} \rho_{\ee} + \ee(t \cos\phi +\rho_{\ee}   s \sin\phi)^2  s\sin\phi }\,d\mu(s)=0,
\end{eqnarray}
where $\rho_{\ee} = \rho_{\ee}(t,\phi,s)$.

When $\ee=0$, we have $K_{\ee}=B^n$, and from (\ref{rhonew}), $\rho_0=\sqrt{1-t^2}$.  Then differentiation of (\ref{first_deriv}) with $\ee=0$, with respect to $t$, yields
$$A_{K_{\ee},\phi}''(t)=-\omega_{n-2}\int_{-1}^1(1-t^2)^{(n-3)/2}
\left(1-\frac{(n-3)t^2}{1-t^2}\right)\,d\mu(s)<0$$
when $n=3$ and when $n\ge 4$ and $|t|<1/\sqrt{n-2}$.  Since $\rho_{\ee}$ and all its derivatives are continuous with respect to $\ee$, we have $A_{K_{\ee},\phi}''(t)<0$ for all sufficiently small $\ee>0$, $|t|<1/(2\sqrt{n-2})$, say, and $\phi\in S^1$. Now for sufficiently small $\ee>0$, we have $|t_{\ee}(\phi)|<1/(2\sqrt{n-2})$, where $t_{\ee}(\phi)$ was defined in (i).  Hence, by the implicit function theorem, for sufficiently small $\ee>0$, the unique $t$ with $|t|<1/(2\sqrt{n-2})$ that satisfies (\ref{first_deriv}) is $t=t_{\ee}(\phi)$, and moreover $t_{\ee}(\phi)\in C^{\infty}(S^1)$.  Summarizing, we have shown that there is a positive $\ee_1(n)\le \delta_1$ such that for all $0<\ee<\ee_1(n)$, $t=t_{\ee}(\phi)$ satisfies (\ref{first_deriv}) and $t_\ee(\phi)\in C^{\infty}(S^1)$.

Let $0<\ee<\ee_1(n)$.  Rearranging (\ref{first_deriv}), we obtain
\begin{equation}\label{newt}
t_\ee(\phi)=-\frac{\int_{-1}^1\Phi(t_{\ee}(\phi),\phi,s)\,d\mu(s)}
{\int_{-1}^1\Psi(t_{\ee}(\phi),\phi,s)\,d\mu(s)}\ee=T_{\ee}(\phi)\ee,
\end{equation}
say, where
\begin{equation}\label{I1}
\Phi=\frac{\rho_{\ee}^{n-2}\,(t_\ee(\phi)\cos\phi +\rho_{\ee} s \sin\phi)^2\cos\phi}{(t_\ee(\phi)^2+\rho_{\ee}^2)^{1/2} \rho_{\ee} + \ee(t_\ee(\phi) \cos\phi +\rho_{\ee}   s\sin\phi)^2  s \sin\phi}
\end{equation}
and
\begin{equation}\label{I2}
\Psi=\frac{\rho_{\ee}^{n-2}\, (t_\ee(\phi)^2+\rho_{\ee}^2)^{1/2}}{(t_\ee(\phi)^2+\rho_{\ee}^2)^{1/2} \rho_{\ee} + \ee(t_\ee(\phi)\cos\phi +\rho_{\ee}s \sin\phi)^2 s \sin\phi}.
\end{equation}
Here $\rho_{\ee} = \rho_{\ee}(t_\ee(\phi),\phi,s)$ and $\Phi$ and $\Psi$ also involve $t=t_{\ee}(\phi)$.  Nevertheless, we have shown that $t_\ee(\phi)=T_{\ee}(\phi)\ee$, say, where $T_{\ee}(\phi)\in C^{\infty}(S^1)$ because $t_{\ee}(\phi)\in C^{\infty}(S^1)$.  This completes the proof of (ii).

$\mathrm{(iii)}$  Let $0<\ee<\ee_1(n)$. In view of (\ref{newt}), (\ref{I1}), and (\ref{I2}) and the previously established bounds $|t_{\ee}(\phi)|<1/4$ and $1/2<\rho_{\ee}<3/2$ that hold when $0<\ee<\ee_1(n)$, there is a constant $c_1(0,n)$ such that $|T_{\ee}(\phi)|\le c_1(0,n)$ for all $\phi\in S^1$.  This proves (\ref{ds}) for $k=0$.

To prove (\ref{ds}) for $k=1$, note that from (\ref{newt}) we obtain
$$T_{\ee}(\phi) \int_{-1}^1  \Psi(T_{\ee}(\phi)\ee,\phi,s)\,d\mu(s)  + \int_{-1}^1 \Phi(T_{\ee}(\phi)\ee,\phi,s )\,d\mu(s)=0.$$
Differentiating this equation with respect to $\phi$ and denoting by $\Phi_i$ and $\Psi_i$ the derivatives of $\Phi$ and $\Psi$ with respect to the $i$th variable, $i=1,2$, we get
$$
T_{\ee}'(\phi)\int_{-1}^1\Psi\,d\mu(s) +T_{\ee}(\phi)\int_{-1}^1  (\Psi_1 T_{\ee}'(\phi)\ee+\Psi_2)\,d\mu(s)+ \int_{-1}^1(\Phi_1 T_{\ee}'(\phi)\ee+ \Phi_2 )  \,d\mu(s)=0.
$$
This yields
\begin{equation}\label{t'}
T_{\ee}'(\phi)=-\frac{\int_{-1}^1\left(T_{\ee}(\phi)\Psi_2+\Phi_2\right)\,d\mu(s)}
{\int_{-1}^1 \left(\Psi+\ee(T_{\ee}(\phi)\Psi_1  + \Phi_1)\right)\,d\mu(s)}.
\end{equation}
We have already shown that $T_{\ee}(\phi)$, $\rho_{\ee}$, and the partial derivatives of $\rho_{\ee}$ with respect to $\phi$ of any order are each bounded uniformly in $\ee$ for $0<\ee<\ee_1(n)$.  Furthermore we have that $\rho_{\ee}$ is bounded below by a positive constant, uniformly in $\ee$. It then follows from (\ref{I1}) and (\ref{I2}) that the partial derivatives $\Phi_i$ and $\Psi_i$, $i=1,2$, are all bounded uniformly in $\ee$.  From (\ref{I2}) we also deduce that $\Psi$ is bounded below by a positive constant, uniformly in $\ee$. These facts and (\ref{t'}) show that $T_{\ee}'(\phi)$ is bounded uniformly in $\ee$, proving (\ref{ds}) for $k=1$.

Differentiation of (\ref{I1}) and (\ref{I2}) with respect to the variables $t_{\ee}(\phi)$ and $\phi$, together with the previously established facts just mentioned, reveals that the partial derivatives of $\Phi(t_{\ee}(\phi),\phi,s)$ and $\Psi(t_{\ee}(\phi),\phi,s)$ with respect to its first and second variables, of any order, are each bounded uniformly in $\ee$ for $0<\ee<\ee_1(n)$.  This and the fact that $\Psi$ is bounded below by a positive constant, uniformly in $\ee$, allow the proof of (iii) to be completed by repeatedly differentiating (\ref{t'}) with respect to $\phi$ and using induction on $k$.
\end{proof}

\begin{lem}\label{lem3}
Let $n\ge 3$, let $K_{\ee}$ be as in Lemma~\ref{lem1}, and let $0<\ee<\ee_1(n)$.  Then $m_{K_\ee}\in C^{\infty}_e(S^{n-1})$ and there is a rotationally symmetric function $g_{\ee}\in C^{\infty}_e(S^{n-1})$ such that
\begin{equation}\label{LL}
(n-1)(R^{-1}m_{K_{\ee}})(u)=1+(R^{-1}g_\ee)(u)\ee,
\end{equation}
for $u\in S^{n-1}$, where $R$ is the spherical Radon transform.  Moreover, there is a constant $c_2(k,n)$ such that
\begin{equation}\label{gg}
\left|\frac{d^k g_\ee(\phi)}{d\phi^k}\right|\le c_2(k,n),
\end{equation}
for all $0<\ee<\ee_1(n)$ and $\phi \in S^1$.
\end{lem}

\begin{proof}
Let $0<\ee<\ee_1(n)$.  In the notation of the previous lemma, using (\ref{A}), we have
\begin{equation}\label{AA}
m_{K_{\ee}}(\phi)=\max_{t\in \mathbb R} A_{K_{\ee},\phi}(t)=A_{K_{\ee},\phi}(t_\ee(\phi)) = \frac{\omega_{n-2}}{n-1}\int_{-1}^1 \rho_{\ee}(t_\ee(\phi),\phi,s)^{n-1}\,d\mu(s),
\end{equation}
for $\phi\in S^1$. We know from the paragraph after (\ref{rhonew}) that  $\rho_{\ee}(t,\phi,s)$ is infinitely differentiable with respect to each of its variables and also that $t_{\ee}(\phi)\in C^{\infty}(S^1)$ by Lemma~\ref{lem2}(ii). It follows that $m_{K_{\ee}}(\phi)\in C^{\infty}_e(S^1)$ and therefore $m_{K_{\ee}}(u)\in C^{\infty}_e(S^{n-1})$.

We wish to rewrite the previous equation in a more convenient form.  To this end, we first use (\ref{rhonew}) to obtain
$$
1-\rho_{\ee} =\frac{(t \cos\phi +\rho_{\ee} s\sin\phi)^3(1+(t^2+\rho_{\ee}^2)^{3/2})}
{(1+\rho_{\ee})(1+t^2+\rho_{\ee}^2+(t^2+\rho_{\ee}^2)^2)}\ee+\frac{t^2}{1+\rho_{\ee}}.
$$
When $t=t_\ee(\phi)$, this and (\ref{t}) imply that
$$\rho_{\ee}(t_\ee(\phi),\phi,s)=1-U_{\ee}(\phi,s)\ee,$$
say, where
\begin{equation}\label{U}
U_{\ee}(\phi,s)=\frac{\left(T_\ee(\phi)\ee\cos\phi +\rho_{\ee} s\sin\phi\right)^3\left(1+(T_\ee(\phi)^2\ee^2+\rho_{\ee}^2)^{3/2}\right)}
{(1+\rho_{\ee})\left(1+T_\ee(\phi)^2\ee^2+\rho_{\ee}^2+
\left(T_\ee(\phi)^2\ee^2+\rho_{\ee}^2\right)^2\right)}+
\frac{T_\ee(\phi)^2\ee}{1+\rho_{\ee}}.
\end{equation}
Therefore
\begin{equation}\label{fe}
\rho_{\ee}(t_\ee(\phi),\phi,s)^{n-1} = 1+f_\ee(\phi,s)\ee,
\end{equation}
say, where
\begin{equation}\label{ff}
f_\ee(\phi,s)=\sum_{j=1}^{n-1}(-1)^j\binom{n-1}{j} U_{\ee}(\phi,s)^j\ee^{j-1}.
\end{equation}
By Lemma~\ref{lem2}(iii), $T_{\ee}(\phi)$ and all its derivatives are bounded above uniformly in $\ee$.  Also, we showed in the proof of Lemma~\ref{lem2}(ii) (paragraph after (\ref{rhonew})) that $\rho_{\ee}$ and all its derivatives with respect to $\phi$ are bounded uniformly in $\ee$, and that $\rho_{\ee}$ is bounded below by a positive constant, uniformly in $\ee$.  These facts and repeated differentiation with respect to $\phi$ of (\ref{ff}) and (\ref{U}) show that for every $k=0,1,\dots$, there is a constant $a(k,n)$ such that
\begin{equation}\label{f}
\left|\frac{d^k f_\ee(\phi,s)}{d\phi^k}\right|\le a(k,n),
\end{equation}
for all $0<\ee<\ee_1(n)$, all $\phi\in S^1$, and all $-1\le s\le 1$.

From (\ref{AA}) and (\ref{fe}), we obtain
\begin{equation}\label{N1}
A_{K_{\ee},\,\phi}(t_\ee(\phi))= \frac{\omega_{n-2}}{n-1}\int_{-1}^1\left(1 +f_\ee(\phi, s) \ee \right)\,d\mu(s).
\end{equation}
Note that when $\ee=0$, $K_{\ee}=B^n$, $t_{\ee}(\phi)=0$, and $A_{K_{\ee},\,\phi}(t_\ee(\phi))=\kappa_{n-1}$.  Hence, or by direct integration, we obtain
\begin{equation}\label{N2}
\frac{\omega_{n-2}}{n-1}\int_{-1}^11\,d\mu(s)=\kappa_{n-1}.
\end{equation}
Let
\begin{equation}\label{gdef}
g_\ee(u)=g_\ee(\phi)=\omega_{n-2}\int_{-1}^1 f_\ee(\phi,s)\,d\mu(s),
\end{equation}
for $u\in S^{n-1}$ and $\phi\in S^1$.  From (\ref{N1}) and (\ref{N2}) we obtain
$$
m_{K_\ee}(u)=\max_{t\in \mathbb R}A_{K_\ee,u}(t)=\kappa_{n-1}+g_\ee(u)\ee/(n-1),
$$
for all $u\in S^{n-1}$.  Now $g_{\ee}(u)\in C^{\infty}_e(S^{n-1})$ follows from $m_{K_{\ee}}(u)\in C^{\infty}_e(S^{n-1})$.  We also have
$$(n-1)(R^{-1}m_{K_{\ee}})(u)=R^{-1}\left((n-1)\kappa_{n-1}\right)
+(R^{-1}g_\ee)(u)\ee=1+(R^{-1}g_\ee)(u)\ee,$$
for all $u\in S^{n-1}$.  Finally, (\ref{gg}) follows directly from (\ref{f}) and (\ref{gdef}).
\end{proof}

\begin{lem}\label{lem4}
Let $p\in \{1,\dots,n-1\}$ and let $g\in C^{\infty}(\R^n\setminus\{o\})$ be homogeneous of degree $-n+p$.  Let
$$
G_{u}(z)=(1-z^2)^{(n-3)/2}\int_{S^{n-1}\cap u^\perp} g(zu +\sqrt{1-z^2}\,v) \,dv,
$$
for $-1\le z\le 1$.  If $g$ is even, then
$$
\widehat{g}(u)=\left\{\begin{array}{ll} (-1)^{(p-1)/2}\pi G_{u}^{(p-1)}(0), &
\mbox{if $p$ is odd,}\\
(-1)^{p/2}(p-1)!\left(
 \int_{-1}^1|z|^{-p}\Bigl(G_{u}(z)-\sum_{k=0}^{p-1}G_{u}^{(k)}(0)z^k/k!\right)\,dz+\\
 +2\sum\{G_{u}^{(k)}(0)/(k!(1+k-p)): k=0,\dots, p-2, k~\text{even}\}\Bigr), & \mbox{if $p$ is even,}
\end{array}\right.
$$
for $u\in S^{n-1}$.  If $g$ is odd, then
$$
\widehat{g}(u)=\left\{\begin{array}{ll} (-1)^{p/2}\pi G_{u}^{(p-1)}(0)i, &
\mbox{if $p$ is even,}\\
(-1)^{(p+1)/2}(p-1)!\left(
 \int_{-1}^1|z|^{-p}\,\sgn z\Bigl(G_{u}(z)-\sum_{k=0}^{p-1}G_{u}^{(k)}(0)z^k/k!\right)\,dz+\\
 +2\sum\{G_{u}^{(k)}(0)/(k!(1+k-p)):k=1,\dots, p-2, k~\text{odd}\}\Bigr)i, & \mbox{if $p$ is odd,}
\end{array}\right.
$$
for $u\in S^{n-1}$.
\end{lem}

\begin{proof}
The Fourier transform of a real even function is real and the Fourier transform of a real odd function is purely imaginary.  Using these well-known facts, the four formulas for $\widehat{g}(u)$ in the statement of the lemma are obtained by taking real or imaginary parts, as appropriate, in the more general formulas \cite[Equations (3.6) and (3.7), p.~395]{GYY}, which apply to any function $g\in C^{\infty}(S^{n-1})$, extended to a homogeneous of degree $-n+p$ function, $p\in \{1,\dots,n-1\}$, on $\R^n\setminus\{o\}$ as in (\ref{homog}) with $f=g$.
\end{proof}

\begin{lem}\label{lem5}
Let $n\ge 3$ and let $K_{\ee}$ be as in Lemma~\ref{lem1}.  For $0<\ee<\ee_1(n)$, define
\begin{equation}\label{defineL}
\rho_{L_{\ee}}(u)=\left((n-1)(R^{-1}m_{K_{\ee}})(u)\right)^{1/(n-1)},
\end{equation}
for all $u\in S^{n-1}$. Then there is a positive $\ee_2(n)\le \ee_1(n)$ such that $\rho_{L_{\ee}}$ is the radial function of an origin-symmetric star body $L_{\ee}$ in $\R^n$ with $\rho_{L_{\ee}}\in C^{\infty}(S^{n-1})$, for all $0<\ee<\ee_2(n)$.
\end{lem}

\begin{proof}
Let $0<\ee<\ee_1(n)$ and let $g_{\ee}\in C^{\infty}_e(S^{n-1})$ be the function from Lemma~\ref{lem3}.  Extend $g_{\ee}$ to a homogeneous of degree $-1$ function in $C^{\infty}_e(\R^n\setminus\{o\})$ as in (\ref{homog}) with $f=g_{\ee}$ and $p=n-1$.  Then, by (\ref{RinvF}), we have
$$(R^{-1}g_\ee)(u) =\frac{\pi}{(2\pi)^n}\widehat{g_{\ee}}(u),$$
for $u\in S^{n-1}$.  From this, (\ref{LL}), and (\ref{defineL}), we conclude that
\begin{equation}\label{newL}
\rho_{L_{\ee}}(u)=\left(1+\frac{\pi}{(2\pi)^n}\widehat{g_{\ee}}(u)\ee\right)^{1/(n-1)},
\end{equation}
for all $u\in S^{n-1}$.

Define
\begin{equation}\label{Gu}
G_{\ee,u}(z)=(1-z^2)^{(n-3)/2}\int_{S^{n-1}\cap u^\perp} g_{\ee}(zu +\sqrt{1-z^2}\,v) \,dv,
\end{equation}
for $-1\le z\le 1$.  Since (\ref{gg}) holds for all $\phi\in S^1$ (and not just for $\phi\in [0,\pi]$), it is not difficult to check that the rotationally symmetric function $g_{\ee}(x/|x|)$, $x\in \R^n\setminus \{o\}$, has each of its partial derivatives with respect to $x_i$, $i=1,\dots,n$, of any order bounded on $S^{n-1}$, uniformly in $\ee$.  From this and (\ref{Gu}) we see that for $k=0,1,\dots$, there are constants $c_3(k,n)$ such that
\begin{equation}\label{Gbound}
|G_{\ee, u}^{(k)}(0)|\le c_3(k,n),
\end{equation}
for all $0<\ee<\ee_1(n)$ and $u\in S^{n-1}$.

Next, we apply Lemma~\ref{lem4} for the case when $g$ is even (with $g$ and $G_{u}$ replaced by $g_{\ee}$ and $G_{\ee,u}$, respectively).  Suppose first that $n$ is even.  By (\ref{Gbound}) and Lemma~\ref{lem4} with $p=n-1$, it follows that $\widehat{g_{\ee}}(u)=(-1)^{(n-2)/2}\pi G_{\ee,u}^{(n-2)}(0)$ is bounded on $S^{n-1}$.  From (\ref{newL}) we conclude that there is a positive $\ee_2(n)\le \ee_1(n)$ such that $\rho_{L_{\ee}}(u)>0$ for all $u\in S^{n-1}$, and hence such that $L_{\ee}$ is an origin-symmetric star body, for all $0<\ee<\ee_2(n)$.

Now suppose that $n$ is odd. By Lemma~\ref{lem4} with $p=n-1$, we obtain $\widehat{g_{\ee}}(u)=(-1)^{(n-1)/2}(n-2)!(I_{\ee}(u)+\Sigma_{\ee}(u))$, where
$$I_{\ee}(u)=\int_{-1}^1|z|^{-n+1}\left(G_{\ee,u}(z)
-\sum_{k=0}^{n-2}G_{\ee,u}^{(k)}(0)\frac{z^k}{k!}\right)\,dz
$$
and
$$\Sigma_{\ee}(u)=2\sum\left\{\frac{G_{\ee,u}^{(k)}(0)}{k!(2+k-n)}:k=0,\dots,n-3, k~{\text{even}}\right\}.$$
By (\ref{Gbound}), $\Sigma_{\ee}(u)$ is bounded on $S^{n-1}$.  The integral $I_{\ee}(u)$ over the range $[0,1]$ can be written as
\begin{equation}\label{split}
I_{\ee}^+(u)=\int_{0}^{1/2}z^{-n+1}\left(G_{\ee, u}(z)-\sum_{k=0}^{n-2}G_{\ee, u}^{(k)}(0)\frac{z^k}{k!}\right)\,dz+\int_{1/2}^{1}z^{-n+1}\left(G_{\ee, u}(z)-\sum_{k=0}^{n-2}G_{\ee, u}^{(k)}(0)\frac{z^k}{k!}\right)\,dz.
\end{equation}
Using (\ref{Gbound}) again, we see that the second integral in (\ref{split}), as a function of $u$, is bounded on $S^{n-1}$.  By Taylor's theorem, the first integral in (\ref{split}) is equal to
$$\int_{0}^{1/2}\frac{G_{\ee, u}^{(n-1)}(c(z))}{(n-1)!} \,dz,$$
for some $0<c(z)<1/2$.  It follows from (\ref{Gu}) that as a function of $u$, $G_{\ee, u}^{(n-1)}(z)$ is bounded on $S^{n-1}$, uniformly in $z$ for $0\le z\le 1/2$.  Therefore $I_{\ee}^+(u)$ is bounded on $S^{n-1}$.  Similarly, the integral $I_{\ee}(u)$ over the range $[-1,0]$ is bounded on $S^{n-1}$.  Consequently, $I_{\ee}(u)$ itself and therefore $\widehat{g_{\ee}}(u)$ are bounded on $S^{n-1}$.  Then, as for the case when $n$ is even, we conclude from (\ref{newL}) that there is a positive $\ee_2(n)\le \ee_1(n)$ such that $L_{\ee}$ is an origin-symmetric star body, for all $0<\ee<\ee_2(n)$.
\end{proof}

\begin{lem}\label{lem6}
Let $L_{\ee}$ be as in Lemma~\ref{lem5}, where $0<\ee<\ee_2(n)$.  Then there is a positive $\ee_3(n)\le \ee_2(n)$ such that $L_{\ee}$ is convex for all $0<\ee<\ee_3(n)$.
\end{lem}

\begin{proof}
Let $0<\ee<\ee_2(n)$ and let $g_{\ee}$ be as in Lemma~\ref{lem5}.  Since $L_\ee$ is a body of revolution whose radial function is given by (\ref{newL}), its intersection with the $x_1,x_n$-plane is a curve whose radial function is
\begin{equation}\label{rp}
r(\phi)= \left(1+\frac{\pi}{(2\pi)^n}\widehat{g_{\ee}}(\sin\phi\,e_1+\cos\phi\,e_n)\,\ee\right)
^{1/(n-1)}.
\end{equation}
It suffices to prove that this curve is convex for small $\ee$.

We claim that the first and second partial derivatives of $\widehat{g_{\ee}}$ with respect to $x_1$ and $x_n$ are bounded on $S^{n-1}$ uniformly in $\ee$ for small $\ee$.  Once this has been proved, it follows that the first and second derivatives of $\widehat{g_{\ee}}(\sin\phi\,e_1+\cos\phi\,e_n)$ with respect to $\phi$ are bounded uniformly in $\ee$ for small $\ee$.  The required convexity is then an easy consequence of the polar coordinate formula (\ref{cpolar}) for curvature, (\ref{rp}), and the fact that $r(\phi)$ is close to 1 when $\ee$ is small.

It remains to prove the claim above. By \cite[Lemma 3.16]{Ko}, $\widehat{g_{\ee}}$ is an infinitely differentiable function that is homogeneous of degree $-n+1$ on $\R^n\setminus \{o\}$.  Therefore for $q=1,\dots,n$, $\partial  \widehat{g_{\ee}}/\partial x_q$ is an infinitely differentiable function that is homogeneous of degree $-n$ on $\R^n\setminus \{o\}$.

Let $\psi$ be a test function with support in $\R^n\setminus\{o\}$ and let $q\in \{1,\dots,n\}$.  Denote by $\langle \partial\widehat{g_{\ee}}/\partial x_q,\psi\rangle$ the action of the distribution $\partial\widehat{g_{\ee}}/\partial x_q$ and note that since $\psi$ has support in $\R^n\setminus\{o\}$, the derivative $\partial\widehat{g_{\ee}}/\partial x_q$ can be regarded in both the distributional and the usual sense.  Let $h_{\ee,q}(x)=x_q g_{\ee}(x)$ for $x\in \R^n\setminus\{o\}$.  Using first (\ref{FD2}) with $f=g_{\ee}$ and then (\ref{Lap}) with $f=h_{\ee,q}$, we obtain
\begin{eqnarray*}
\left\langle \frac{\partial\widehat{g_{\ee}}}{\partial x_q},\psi\right\rangle &=&
-i\langle\widehat{h_{\ee,q}},\psi\rangle= - i\langle |x|^2\widehat{h_{\ee,q}}(x), |x|^{-2}\psi(x)\rangle\\
&=& i\langle\widehat{\Delta h_{\ee,q}}(x),|x|^{-2}\psi(x)\rangle
=i\langle|x|^{-2}\widehat{\triangle h_{\ee,q}}(x),\psi(x) \rangle.
\end{eqnarray*}
This gives
\begin{equation}\label{distr}
\frac{\partial \widehat{g_{\ee}}}{\partial x_q}=i |x|^{-2}\widehat{\triangle h_{\ee,q}}
\end{equation}
on $\R^n\setminus\{o\}$, because by \cite[Theorem~6.25]{Rud}, distributions that are equal on test functions with support in $\R^n\setminus\{o\}$ can differ at most by derivatives of the delta function.

Let $g(x)=\triangle h_{\ee,q}(x)=\triangle(x_q g_{\ee})(x)$ for $x\in \R^n\setminus\{o\}$.  Then, with (\ref{distr}) in hand, the desired bounds for $\partial  \widehat{g_{\ee}}/\partial x_q$ on $S^{n-1}$ will follow from corresponding bounds for $\widehat{g}$.  To this end, observe that since $x_q g_{\ee}$ is an infinitely differentiable odd function on $\R^n\setminus \{o\}$ that is homogeneous of degree $0$, $g$ is an infinitely differentiable odd function on $\R^n\setminus \{o\}$ that is homogeneous of degree $-2$.  Therefore $\widehat{g}$ is an infinitely differentiable function on $\R^n\setminus \{o\}$ that is homogeneous of degree $-n+2$ and which can be computed by Lemma \ref{lem4}.  Letting
$$
G_{u}(z)=(1-z^2)^{(n-3)/2}\int_{S^{n-1}\cap u^\perp} g(zu +\sqrt{1-z^2}\,v) \,dv,
$$
for $-1\le z\le 1$, and using (\ref{gg}), we see that for $k=0,1,\dots$, there are constants $c_4(k,n)$ such that
\begin{equation}\label{lastb}
|G_{u}^{(k)}(0)|\le c_4(k,n),
\end{equation}
for all $0<\ee<\ee_2(n)$ and $u\in S^{n-1}$.  Bounds for $\widehat{g}$ now follow from Lemma~\ref{lem4} for the case when $g$ is odd and $p=n-2$.  Indeed, it follows immediately from this and (\ref{lastb}) that when $n$ is even, $\widehat{g}(u)=(-1)^{(n-2)/2}\pi G_{u}^{(n-3)}(0)$ is bounded on $S^{n-1}$ uniformly in $\ee$ for small $\ee$.  The case when $n$ is odd is handled in exactly the same way as in Lemma~\ref{lem5} and we omit the details.

Bounds for the second partial derivatives of $\widehat{g_{\ee}}$ on $S^{n-1}$ are obtained in an entirely analogous fashion, as follows.   Let $q,r\in \{1,\dots,n\}$ and let $j_{\ee,q,r}(x)=x_q x_r g_{\ee}(x)$ for $x\in \R^n\setminus\{o\}$. An argument similar to the one that showed (\ref{distr}) gives
$$\frac{\partial^2 \widehat{g_{\ee}}}{\partial x_q\partial x_r}=|x|^{-2}\widehat{\triangle j_{\ee,q,r}}$$
on $\R^n\setminus\{o\}$. Here $\widehat{\triangle j_{\ee,q,r}}$  is an infinitely differentiable even function on $\R^n\setminus \{o\}$ that is homogeneous of degree $-1$.  Bounds for $\widehat{\triangle j_{\ee,q,r}}$ on $S^{n-1}$ are then obtained via (\ref{lastb}) and Lemma~\ref{lem4} for the case when $g$ is even and $p=n-1$.  Again, we omit the details since the argument is exactly the same as for the first partial derivatives of $\widehat{g_{\ee}}$.
\end{proof}

\noindent{\em{Proof of Theorem~\ref{mainthm}}.} Let $n\ge 3$ and $0<\ee<\ee_3(n)$.  Let $K=K_{\ee}$ be defined by (\ref{rho}) and let $L=L_{\ee}$ be defined by (\ref{defineL}).  By Lemmas~\ref{lem1} and~\ref{lem6}, $K$ and $L$ are convex bodies of revolution about the $x_n$-axis.  Lemma~\ref{lem1} shows that $\rho_K\in C^{\infty}(S^{n-1})$ and that $K$ is not centrally symmetric.  From Lemma~\ref{lem5}, we get that $\rho_L\in C^{\infty}(S^{n-1})$ and that $L$ is origin symmetric.  The Brunn-Minkowski theorem implies that $m_L(u)=V(L\cap u^{\perp})$ for all $u\in S^{n-1}$. By Lemma~\ref{lem3}, we have $m_K\in C^{\infty}_e(S^{n-1})$.  Therefore, by (\ref{defineL}) and the fact that $R:C_e^{\infty}(S^{n-1})\rightarrow C_e^{\infty}(S^{n-1})$ is a continuous bijection, we obtain
$$m_L(u)=V(L\cap u^{\perp})=\frac{1}{n-1}(R\rho_L^{n-1})(u)=(R(R^{-1}m_K))(u)=m_K(u),$$
for all $u\in S^{n-1}$. \qed

\section{Concluding remarks}\label{conclude}

The definition (\ref{rho}) of the body $K=K_{\ee}$ appears to be essentially the simplest that allows our construction to work.  For example, if we define
$$
\rho_{K_{\ee}}(\phi)=\left(1+\ee \cos\phi\right)^{-1},
$$
where $0<\ee<1$ and $0\le \phi\le \pi$ is the angle with the positive $x_n$-axis, then $K=K_{\ee}$ is an ellipsoid.  In particular, it is centrally symmetric.  If $K'$ denotes the origin-symmetric translation of $K$, then our construction would yield $IL=CK=CK'=IK'$.  This would imply $K'=L$ and hence that $K$ and $L$ are equal up to a translation.  Of course the formula
$$
\rho_{K_{\ee}}(\phi)=\left(1+\ee\cos^2\phi\right)^{-2},
$$
gives a $K=K_{\ee}$ that is already origin symmetric.

Klee \cite{Kl}, \cite{Kl2} (see also \cite[p.~24]{CFG} and \cite[Problem 8.8(ii)]{G}) also asked whether a convex body in $\R^n$, $n\ge 3$, whose inner section function is constant, must be a ball.  The hypothesis is actually weaker than that of Bonnesen's question in \cite{B} of whether a convex body in $\R^n$, $n\ge 3$, must be a ball if both its inner section function and its brightness function are constant.  Today, Bonnesen's question (see \cite[p.~24]{CFG} and \cite[Problem 8.9(ii)]{G}) is one of the oldest open problems in convex geometry.  (For $n=2$, a counterexample is again provided by any convex body of constant width that is not a disk; this follows from a result of E.~Makai and H.~Martini (see \cite[Theorem~8.3.5]{G}) that in the plane, the cross-section body and projection body coincide.)

\end{document}